\documentclass[11pt,leqno]{article}
\usepackage[margin=1in]{geometry}
\usepackage{amssymb,amsfonts,amsmath,bbm,mathrsfs,stmaryrd,mathtools}
\usepackage{xcolor}
\usepackage{url}

\usepackage[shortlabels]{enumitem}
\usepackage{tensor}

\usepackage[backref=page]{hyperref}
\hypersetup{colorlinks,
             linkcolor=black!75!red,
             citecolor=blue,
             pdftitle={},
             pdfproducer={pdfLaTeX},
             pdfpagemode=None,
             bookmarksopen=true
             bookmarksnumbered=true}

\usepackage{tikz}
\usetikzlibrary{arrows,calc,decorations.pathreplacing,decorations.markings,intersections,shapes.geometric,through,fit,shapes.symbols,positioning,decorations.pathmorphing}

\usepackage{braket}
\usepackage{tensor}

\usepackage[amsmath,thmmarks,hyperref]{ntheorem}
\usepackage[matrix, arrow, curve]{xy}

\theoremstyle{plain}

\newtheorem{lemma}{Lemma}[section]
\newtheorem{proposition}[lemma]{Proposition}
\newtheorem{corollary}[lemma]{Corollary}
\newtheorem{theorem}[lemma]{Theorem}

\theoremstyle{plain}
\theoremnumbering{Alph}
\newtheorem{theoremN}{Theorem}

\theoremstyle{plain}
\theorembodyfont{\upshape}
\theoremsymbol{\ensuremath{\blacklozenge}}

\newtheorem{remark}[lemma]{Remark}
\newtheorem{remarks}[lemma]{Remarks}

\newtheorem{construction}[lemma]{Construction}

\def\NN{{\mathbb N}}

\numberwithin{equation}{section}

\theoremstyle{nonumberplain}
\theoremsymbol{\ensuremath{\blacksquare}}

\newtheorem{proof}{Proof}

\newcommand\cC{{\mathcal C}}
\newcommand\cD{{\mathcal D}}

\DeclareMathOperator{\id}{id}

\newcommand{\cat}[1]{\textsc{#1}}

\newcommand{\qedhere}{\mbox{}\hfill\ensuremath{\blacksquare}}

\title{On the category of Hopf braces}

\date{}

\author{Ana Agore and Alexandru Chirvasitu}

\begin{document}

\newcommand{\Addresses}{{
    \bigskip
    \footnotesize


    \textsc{Simion Stoilow Institute of Mathematics of the Romanian Academy}
    \par\nopagebreak
    \textsc{Calea Grivitei 21, Bucharest 010702, Romania \, \&}
    \par\nopagebreak
    \textsc{Vrije Universiteit Brussel}
    \par\nopagebreak
    \textsc{Pleinlaan 2, Brussels 1050, Belgium\, \&}
    \par\nopagebreak
    \textsc{Max Planck Institute for Mathematics}
     \par\nopagebreak
     \textsc{Vivatsgasse 7, 53111 Bonn, Germany}
       \par\nopagebreak
    \textit{E-mail address}: \texttt{ana.agore@gmail.com, ana.agore@vub.be}
    
    \medskip
    
    \textsc{Department of Mathematics, University at Buffalo}
    \par\nopagebreak
    \textsc{Buffalo, NY 14260-2900, USA}
    \par\nopagebreak
    \textit{E-mail address}: \texttt{achirvas@buffalo.edu}
  }}

\maketitle

\begin{abstract}
  Hopf braces are the quantum analogues of skew braces and, as such, their cocommutative counterparts provide solutions to the quantum Yang-Baxter equation.  We investigate various properties of categories related to Hopf braces. In particular, we prove that the category of Hopf braces is accessible while the category of cocommutative Hopf braces is even locally presentable. We also show that functors forgetting multiple antipodes and/or multiplications down to coalgebras are monadic. Colimits in the category of cocommutative Hopf braces are described explicitly and a free cocommutative Hopf brace on an arbitrary cocommutative Hopf algebra is constructed.
\end{abstract}


\vspace{.5cm}



\tableofcontents

\section*{Introduction}

Prompted by Drinfeld's' \cite[Section 9]{dr} idea of studying set-theoretic solutions of the Yang-Baxter equation \cite{baxter, Yang}, a plethora of new algebraic structures have been introduced with the purpose of constructing said solutions: e.g., braided groups \cite{MT-mp}, (skew) braces \cite{gv_skbr}, Rota-Baxter operators on groups \cite{MR4370524}, post-groups \cite{BGST-MA}, trusses \cite{MR4009388}, to name but a few. Among these various algebraic structures, skew braces play a central role. Defined as a pair of appropriately compatible groups on the same underlying set, skew braces are used for constructing non-degenerate set-theoretic solutions of the Yang-Baxter equation.  Their quantum analogues, called \emph{Hopf braces} \cite{AGV}, consist of two compatible Hopf algebra structures sharing the same underlying coalgebra (see Section~\ref{prel} for the precise definition). As expected, cocommutative Hopf braces provide solutions to the quantum Yang-Baxter equation \cite[Theorem 2.3]{AGV}.

The present paper focuses on category-theoretic properties of Hopf braces (cocommutative or not), along with those of precursor categories such as those of coalgebras equipped with multiple compatible Hopf or bialgebra structures. Adjacent themes are taken up in \cite{2411.19238v1}, which centers around category-theoretic issues relevant from the point of view of homological algebra.

To unwind some of the specifics of our main results, we recall some vocabulary. 
\begin{itemize}[wide]
\item An object $x\in \cC$ in a category is \emph{$\lambda$-presentable} \cite[Definition 1.13(2)]{ar} for a \emph{regular cardinal} \cite[Definition I.10.34]{kun_st_2nd_1983} $\lambda$ if the representable functor $\cC(x,-)$ preserves \emph{$\lambda$-directed colimits}, i.e. \cite[Definition 1.13(1)]{ar} those indexed by posets in which every subset of cardinality $<\lambda$ as an upper bound.

  \emph{Presentable} objects are those that are $\lambda$-presentable for some regular cardinal. 

\item $\cC$ is \emph{locally $\lambda$-presentable} \cite[Definition 1.17]{ar} for regular $\lambda$ if it is cocomplete and every object is a $\lambda$-directed colimit of $\lambda$-presentable objects.

  $\cC$ is \emph{locally presentable} if it is locally $\lambda$-presentable for some regular cardinal $\lambda$. 
\end{itemize}

Local presentability purchases much good behavior \cite[Remark 1.56]{ar}: completeness follows, objects have sets (as opposed to proper classes) of subobjects and quotient objects, various \emph{adjoint functor theorems} \cite[\S 18]{ahs} apply (e.g. \emph{cocontinuous}, or colimit-preserving functors between such categories are left adjoints \cite[\S V.8, Corollary]{mcl_2e}), and so on. Some of our main results are to the effect that these conditions obtain in various categories of coalgebras with multiple Hopf algebra/bialgebra structures, as well as that of cocommutative Hopf braces of interest in the quantum-Yang-Baxter literature.

Specifically, we amalgamate portions of Theorem \ref{th:ijh.loc.pres}, Corollary \ref{cor:ijh.co.compl}, Theorem \ref{th:hbr.access} and Corollary \ref{cor:hbr.co.compl} in the following statement. All structures in sight are linear over a field $\Bbbk$, while $\cat{Coalg}$ and $\tensor*[_{(I,J)}]{\cat{Halg}}{}$ denote, respectively, the categories of
\begin{itemize}[wide]
\item $\Bbbk$-coalgebras;

\item and $\Bbbk$-coalgebras equipped with $I$-indexed bialgebra structures all sharing the same unit, so that those multiplications indexed by $J\subseteq I$ have antipodes (and hence the $J$-indexed algebras are Hopf). 
\end{itemize}
An additional `coc' superscript indicates cocommutativity.

\begin{theoremN}\label{thn:locpres}
  \begin{enumerate}[(1),wide]
  \item The categories $\tensor*[_{(I,J)}]{\cat{Halg}}{}$ are all locally presentable, so in particular complete and cocomplete.

  \item The same applies to $\tensor*[_{(I,J)}]{\cat{Halg}}{^{coc}}$.

  \item As well as the category $\cat{Hbr}^{coc}$ of cocommutative Hopf braces.

  \item The category $\cat{Hbr}$ of arbitrary Hopf braces is \emph{accessible} \cite[Remark 2.2(1)]{ar}, and hence locally presentable if and only if it is (co)complete.  \qedhere 
  \end{enumerate}
\end{theoremN}
We mention that Theorem \ref{thn:locpres} has seen some use and interest in \cite{2411.19238v1}, where \cite[\S 5]{2411.19238v1} relies on the existence of binary coproducts in $\cat{Hbr}^{coc}$ in deducing that category's \emph{semi-abelianness} \cite[\S 2.5, Definition]{jmt_semiab}. 

A separate but related theme we tackle is that of \emph{monadicity} for functors between categories of multiple Hopf algebras. A few recollections:

\begin{itemize}[wide]
\item A \emph{monad} $(T,\mu,\eta)$ (\cite[Definition 4.1.1]{brcx_hndbk-2}, \cite[\S VI.1]{mcl_2e}) on a category $\cC$ is a unital associative algebra
  \begin{equation*}
    \id\xrightarrow{\quad\eta\quad}T
    ,\quad
    T\circ T\xrightarrow{\quad\mu} T
  \end{equation*}
  in the \emph{monoidal category} \cite[Definition 6.1.1]{brcx_hndbk-2} $\left(\cC^{\cC},\ \circ,\ \id\right)$ of $\cC$-endofunctors with composition as the tensor product. 

\item An \emph{algebra} \cite[Definition 4.1.2]{brcx_hndbk-2} over $(T,\mu,\eta)$, also called a $T$-algebra, is an object $c\in \cC$ equipped with a morphism $Tc\to c$ satisfying the unitality and associativity conditions with respect to $\mu$ and $\eta$ reminiscent of the usual definition of a module over an algebra.

  $T$-algebras form a category typically denoted by $\cC^T$, equipped with a forgetful functor back to $\cC$
  
\item A functor $\cD\xrightarrow{F}\cC$ is \emph{monadic} \cite[Definition 4.4.1]{brcx_hndbk-2} (or $\cD$ is \emph{monadic along $F$}) if there is a monad $(T,\mu,\eta)$ on $\cC$ making the following diagram commute (up to natural isomorphism):
  \begin{equation*}
    \begin{tikzpicture}[>=stealth,auto,baseline=(current  bounding  box.center)]
      \path[anchor=base] 
      (0,0) node (l) {$\cD$}
      +(2,.5) node (u) {$\cC^T$}
      +(4,0) node (r) {$\cC$}
      ;
      \draw[->] (l) to[bend left=6] node[pos=.5,auto] {$\scriptstyle \text{equivalence}$} (u);
      \draw[->] (u) to[bend left=6] node[pos=.5,auto] {$\scriptstyle \cat{forget}$} (r);
      \draw[->] (l) to[bend right=6] node[pos=.5,auto,swap] {$\scriptstyle F$} (r);
    \end{tikzpicture}
  \end{equation*}
\end{itemize}
Monadic functors (sometimes also termed \emph{tripleable} \cite[\S 3.3]{bw}) are one way of capturing the intuition that $\cD$ consists, roughly speaking, of objects in $\cC$ equipped with additional ``algebraic structure''. Preeminent examples \cite[\S VI.8, Theorem 1]{mcl_2e} very much in line with this intuition are the forgetful functors back to $\cat{Set}$ from \emph{varieties of algebras}: categories of associative algebras, groups, monoids, commutative rings, and so forth. 

Not only are monadic functors right adjoints, but they reflect isomorphisms, preserve sufficiently well-behaved coequalizers \cite[Theorem 4.4.4]{brcx_hndbk-2}, afford adjunction-lifting results \cite[Theorem 4.5.6]{brcx_hndbk-2}, and similarly lend themselves to any number of other applications. In light of this, the following result might be of some interest in the present context of studying Hopf braces; it again aggregates sub-statements of Theorems \ref{th:ijh.loc.pres} and \ref{th:hbr.access}. 

\begin{theoremN}\label{thn:monad}
  \begin{enumerate}[(1),wide]
  \item For sets $J'\subseteq J\subseteq I$ the forgetful functors
    \begin{equation*}
      \begin{aligned}
        \tensor*[_{(I,J)}]{\cat{Halg}}{}
        &\xrightarrow{\qquad}
          \cat{Coalg}\\
        \tensor*[_{(I,J)}]{\cat{Halg}}{}
        &\xrightarrow{\qquad}
          \tensor*[_{(I,J')}]{\cat{Halg}}{}
      \end{aligned}
    \end{equation*}
    (of which the second forgets the $(J\setminus J')$-indexed antipodes) are monadic. 
    
  \item The forgetful functor 
  \begin{equation*}
      \cat{Hbr}^{coc}
      \xrightarrow{\qquad}
      \cat{Coalg}^{coc}
    \end{equation*}
   from cocommutative Hopf braces to cocommutative coalgebras is monadic.
  \qedhere
  \end{enumerate}
\end{theoremN}

As the original motivation for studying cocommutative Hopf braces lies in the fact that each such gadget generates a solution of the quantum Yang-Baxter equation, being able to produce new examples is a central theme of the theory. With this in mind, Section~\ref{se:colim} provides explicit constructions for various colimits mentioned in Theorem~\ref{thn:locpres}. Specifically, coproducts and coequalizers for cocommutative Hopf braces are explicitly described, starting with the analogous constructions being performed in the categories of $2$-bialgebras and $2$-Hopf algebras, respectively.  
Furthermore, the free cocommutative Hopf brace on an arbitrary cocommutative Hopf algebra is explicitly constructed (see Proposition~\ref{pr:ladj} and Proposition~\ref{pr:incadj}).

\subsection*{Acknowledgements}

The authors are grateful to Marino Gran and Andrea Sciandra for insightful questions which have been the starting point of this work. The second named author was partially supported by a grant of the Ministry of Research, Innovation and Digitization, CNCS/CCCDI -- UEFISCDI, project number PN-IV-P2-2.2-MCD-2024-0354. 

\section{Preliminaries}\label{prel}

Throughout the paper $\Bbbk$ is a field and all vector spaces, tensor products, homomorphisms, (co)algebras, bialgebras, Hopf algebras/braces are linear over $\Bbbk$. For a coalgebra $C$, we use Sweedler's $\Sigma$-notation $\Delta(c)= c_{(1)}\otimes c_{(2)}$ with suppressed summation sign. $A^{op}$ stands for the opposite of the algebra $A$ and, similarly, $C^{cop}$ denotes the co-opposite of the coalgebra $C$. In what follows {$\cat{Vect}$ denotes the  category of vector spaces. Further categories of interest will be introduced in Section~\ref{se:mult.hopf}.

We refer to \cite{rad, Sw} for unexplained notions concerning Hopf algebras and to \cite[Chapter 5]{CMZ} for a thorough background on the quantum Yang-Baxter equation. Category-theoretic terminology, as covered in various sources such as \cite{ahs, ar, brcx_hndbk-1, brcx_hndbk-2}, will feature prominently. More specific references are provided throughout the text.  

Recall \cite[Definition 1.1]{AGV} that a \emph{Hopf brace} over a coalgebra $(H, \Delta, \varepsilon)$ consists of two Hopf algebra structures on $H$, denoted by $(H, \cdot, 1, \Delta, \varepsilon, S)$ and respectively $(H, \diamond, 1_{\diamond}, \Delta, \varepsilon, T)$, compatible in the sense that for all $x$, $y$, $z \in H$ we have:
\begin{equation}\label{eq:hbr}
  x \diamond (y \cdot z) = \bigl(x_{(1)} \diamond y\bigl) \cdot \,S(x_{(2)})\cdot \bigl(x_{(3)} \diamond z\bigl)
\end{equation}
We refer to the two Hopf algebras $(H, \cdot, 1, \Delta, \varepsilon, S)$ and $(H, \diamond, 1_{\diamond}, \Delta, \varepsilon, T)$ underlying a Hopf brace as the \emph{first} and \emph{second Hopf algebra structure} of $H$ and we denote them by $H_{\bullet}$ and $H_{\diamond}$ respectively.  By setting $x = y = 1_{\diamond}$ and respectively $x = z = 1_{\diamond}$ in \eqref{eq:hbr} it is easily seen \cite[Remark 1.3]{AGV} that in any Hopf brace we have $1 = 1_{\diamond}$. A Hopf brace will be called cocommutative if its underlying coalgebra is cocommutative.

Given two Hopf braces $(H, \cdot, \diamond)$ and $(H', \cdot, \diamond)$, a $k$-linear map $f$ between the two underlying vector spaces is called a \emph{morphism of Hopf braces} if both $f \colon H_{\cdot} \to H_{\cdot}^{'}$ and $f\colon H_{\diamond} \to H^{'}_{\diamond}$ are morphisms of Hopf algebras. A coideal $I \subseteq H$ which is a Hopf ideal with respect to both Hopf algebra structures of $H$ will be called a \emph{Hopf brace ideal}. If $I$ is a Hopf brace ideal, then the quotient space $H/I$ is again a Hopf brace and the quotient map $\pi \colon H \to H/I$ becomes a morphism of Hopf braces.

Any Hopf algebra $(H, \cdot, 1, \Delta, \varepsilon, S)$  gives rise to a Hopf brace by considering $x \diamond y = x \cdot y$ for all $x$, $y \in H$. Furthermore, if $(H, \cdot, 1, \Delta, \varepsilon, S)$ is an involutory Hopf algebra, then we can define a new multiplication on $H$ by $x \diamond y = y \cdot x$, for all $x$, $y \in H$, which together with the underlying coalgebra structure of $H$ form a new Hopf algebra. Moreover, it is straightforward to check that \eqref{eq:hbr} is trivially fulfilled and thus $(H, \cdot, \cdot^{op}, 1, \Delta, \varepsilon, S, S)$ is a Hopf brace. If $S$ is only invertible, one obtains a Hopf brace $(H, \cdot, \cdot^{op}, 1, \Delta, \varepsilon, S, S^{-1})$.

The following device will come in handy repeatedly. 

\begin{construction}\label{con:id.chn}
  Let $\bigl(H, \cdot,\diamond, 1, \Delta, \varepsilon, S, T\bigl)$ be a \emph{2-Hopf algebra}, i.e. a coalgebra equipped with two compatible Hopf structures sharing a unit (we examine such structures more closely in Section \ref{se:mult.hopf}). For a subspace $J\le H$, write
  \begin{itemize}[wide]
  \item $J_{1}$ for the ideal generated by $J_{0}:=J$, ${S}(J_{0})$, ${S}^{2}(J_{0})$, $\cdots$ with respect to the first algebra structure of $H$,
  \item $J_{2}$ for the ideal generated by $J_{1}$, ${T}(J_{1})$, ${T}^{2}(J_{1})$, $\cdots$ with respect to the second algebra structure of $H$.
  \end{itemize} 
  Note that ${S}(J_{1}) \subseteq J_{1}$ and ${T}(J_{2}) \subseteq J_{2}$. Continuing this process we obtain a chain 
  \begin{equation}\label{eq:1}
    J=J_{0} \subset J_{1} \subset J_{2} \subset J_{3} \subset \cdots \subset J_{n} \subset \cdots
  \end{equation}
  of subspaces in $H$, and we set $\widetilde{J} := \cup_{n \in \NN} J_{n}$.

  If $J$ is a coideal so are all of the $J_n$, hence also $\widetilde{J}$. Furthermore, we claim that $\widetilde{J}$ is in fact a Hopf brace ideal in the sense of the preceding discussion.

  Indeed, if $h \in H$ and $x \in \widetilde{J}$ then there exists some $n_{0} \in \NN$ such that $x \in J_{n_{0}} \subset J_{n_{0}+1}$. We can assume without loss of generality that $J_{n_{0}}$ and $J_{n_{0}+1}$ are ideals with respect to the first and the second algebra structures of $H$, respectively. Therefore we have $h  {\cdot} x$, $x  {\cdot} h \in J_{n_{0}} \subset \widetilde{J}$ and $h  {\diamond} x$, $x  {\diamond} h \in J_{n_{0}+1} \subset \widetilde{J}$.  Moreover, having in mind the way we defined the chain of coideals \eqref{eq:1} we obviously have $ {S}(x) \in J_{n_{0}} \subset \widetilde{J}$ and $ {T}(x) \in J_{n_{0}+1} \subset \widetilde{J}$.
\end{construction}


\section{Coalgebra-overlapping Hopf structures}\label{se:mult.hopf}

Hopf braces and their set-theoretic counterparts (braces \cite[Definition 2 and Proposition 4]{zbMATH05118810}, skew braces \cite[Definition 1.1]{gv_skbr}), involving multiple group structures on a common set, motivate studying the following type of (co)algebraic structure.

\begin{itemize}[wide]

\item Vector spaces over $\Bbbk$ endowed with two (unital, associative) algebra structures sharing the same unit called \emph{$2$-algebras}. Together with unit preserving $\Bbbk$-linear maps which respect both algebra structures they form a category denoted by $\tensor*[_2]{\cat{Alg}}{}$.

\item Coalgebras over $\Bbbk$, implicitly assumed coassociative and counital, constituting a category $\cat{Coalg}$.

\item We consider also coalgebras equipped with multiple (unital, associative) algebra structures: the unit is common to all such, while the multiplications are indexed by a set $I$. When all these algebra structures are coalgebra maps the resulting gadgets are termed \emph{$I$-bialgebras}, and constitute a category $\tensor*[_I]{\cat{Bialg}}{}$.

\item Those $I$-bialgebras whose $J$-indexed bialgebra structures happen to be Hopf for $J\subseteq I$ we call \emph{$(I,J)$-Hopf algebras}; the category in question is $\tensor*[_{(I,J)}]{\cat{Halg}}{}$.

\item For $J=I$ we streamline the subscript to plain $I$, as in $\tensor*[_I]{\cat{Halg}}{}$. Furthermore, identifying non-negative integers $n$ respectively with $\left\{0,\ \cdots,\ n-1\right\}$, the meaning of symbols such as $\tensor*[_n]{\cat{Bialg}}{}$ or $\tensor*[_n]{\cat{Halg}}{}$ becomes apparent. The unadorned $\cat{Bialg}$ and $\cat{Halg}$ denote the classical categories of bialgebras and Hopf algebras, respectively.

\item We occasionally constrain the coalgebra structures by requiring cocommutativity, in which case the corresponding categories acquire a `coc' superscript: $\cat{Coalg}^{coc}$, $\tensor*[_I]{\cat{Bialg}}{^{coc}}$, etc.
\end{itemize}

\begin{remarks}\label{res:digr}
  \begin{enumerate}[(1),wide]
  \item\label{item:res:digr:qset} Coalgebras are to be regarded, in this context, as a type of ``quantum set'': ordinary sets $S$ are recoverable via their corresponding \emph{group-like coalgebras} $\Bbbk S$ \cite[p.6, Example (1)]{Sw}, defined uniquely by
    \begin{equation*}
      s_1\otimes s_2:=s\otimes s
      \quad\text{and}\quad
      \varepsilon(s)=1
      ,\quad
      \forall s\in S.
    \end{equation*}
    Adopting this perspective, the category $\tensor*[_2]{\cat{Halg}}{}$ (of special interest below because it houses the subcategory of Hopf braces) is a kind of linearized version of the category of structures \cite[\S 3]{zbMATH01477087} (or \cite[Introduction]{zbMATH07673993}) called \emph{digroups}: these are sets equipped with two group structures with a common unit.

  \item\label{item:res:digr:bialg2coalg} We will take for granted the fact that $\tensor*[_I]{\cat{Bialg}}{}$ is \emph{locally presentable} \cite[Definition 1.17]{ar} and the forgetful functor
    \begin{equation}\label{eq:ibialg2coalg}
      \tensor*[_{I}]{\cat{Bialg}}{}
      \xrightarrow{\quad U_I:=\cat{forget}\quad}
      \cat{Coalg}
    \end{equation}
    is \emph{monadic} \cite[Definition 4.4.1]{brcx_hndbk-2} (indeed, \emph{finitary} monadic, in the sense \cite[\S 3.18]{ar} that the monad in question preserves \emph{filtered colimits} \cite[Definition 1.4]{ar}). This is no more difficult to prove than the analogue for $I:=\{*\}$ in \cite[Summary 4.3, bottom right-hand functor]{zbMATH05312006} (taking that discussion's base category $\cC$ to be $\cat{Coalg}$).
  \end{enumerate}  
\end{remarks}

\begin{theorem}\label{th:ijh.loc.pres}
  \begin{enumerate}[(1),wide]
  \item\label{item:th:ijh.loc.pres:all} For set inclusions $J'\subseteq J\subseteq I$ the category $\tensor*[_{(I,J)}]{\cat{Halg}}{}$ is locally presentable and monadic along
  \begin{equation}\label{eq:fgt2coalg}
    \tensor*[_{(I,J)}]{\cat{Halg}}{}
    \xrightarrow{\quad U_{(I,J)}:=\cat{forget}\quad}
    \cat{Coalg}
  \end{equation}
  as well as
  \begin{equation}\label{eq:fgt2bialg}
    \tensor*[_{(I,J)}]{\cat{Halg}}{}
    \xrightarrow{\quad U_{(,J\to J')}:=\cat{forget}\quad}
    \tensor*[_{(I,J')}]{\cat{Halg}}{}.
  \end{equation}

\item\label{item:th:ijh.loc.pres:coc} The analogous statement holds in the cocommutative case. 
  \end{enumerate}  
\end{theorem}
\begin{proof}
  The same argument handles both versions, so we limit ourselves to \ref{item:th:ijh.loc.pres:all}. A number of auxiliary observations will help move the proof along. Much of the argument eventually proving the local presentability of the category of plain Hopf algebras in \cite[Theorem 6]{porst_formal-2} will transport over.
  
  \begin{enumerate}[(I),wide]
  \item\label{item:th:ijh.loc.pres:pf.crlim}: The functor \eqref{eq:fgt2coalg} \emph{creates limits} \cite[Definition 13.17(2)]{ahs}. Recall \cite[Proposition 47 2.]{porst_formal-1} that $\cat{Coalg}$ is locally presentable so in particular complete by \cite[Remark 1.56(1)]{ar}. That the $\cat{Coalg}$-limit of a diagram in $\tensor*[_{(I,J)}]{\cat{Halg}}{}$ acquires a unique compatible $(I,J)$-Hopf algebra structure follows immediately from the analogous statements for the forgetful functor for plain Hopf algebras or bialgebras (the forgetful functors from (cocommutative) Hopf algebras and bialgebras to coalgebras are monadic, see e.g. \cite[Proposition 47 3.]{porst_formal-1} and \cite[Theorem 10 2. and Proposition 28 2.]{porst_formal-2}): simply apply that remark to each individual bialgebra/Hopf algebra structure. 

  \item\label{item:th:ijh.loc.pres:pf.acc}: $\tensor*[_{(I,J)}]{\cat{Halg}}{}$ is \emph{accessible} in the sense of \cite[Remark 2.2(1)]{ar}. The proof strategy is that employed in proving the accessibility of ordinary Hopf algebras (or rather more broadly, Hopf monoids internal to appropriate symmetric monoidal categories) in \cite[Proposition 47 5.]{porst_formal-1}, with minor modifications.

    Observe first that just as for plain bialgebras (e.g. \cite[\S 2.7, Remark 2]{zbMATH05312006} or \cite[\S 2.2, item 2.]{porst_formal-1}), the forgetful functor $\tensor*[_I]{\cat{Bialg}}{}\to \cat{Coalg}$ realizes its domain as a joint \emph{equifier} \cite[Lemma 2.76]{ar} for a family of pairs of natural transformations between \emph{accessible functors} \cite[Definition 2.16]{ar}. The accessibility of $\tensor*[_I]{\cat{Bialg}}{}\to \cat{Coalg}$ follows from the selfsame \cite[Lemma 2.76]{ar} recalling that both $\tensor*[_I]{\cat{Bialg}}{}$ and $\cat{Coalg}$ are locally presentable (hence in particular accessible by \cite[Corollary 2.47]{ar}). 
    
    As to $(I,J)$-Hopf algebras, consider the accessible endofunctor $(-)^{cop}$ on $\cat{Coalg}$. An antipode for the $j$-labeled multiplication of an object $B\in \tensor*[_I]{\cat{Bialg}}{}$ (for $j\in J\subseteq I$) is in particular a \emph{$(-)^{cop}$-algebra structure} in the sense of \cite[Notation 2.74]{ar}, i.e. a morphism $B^{cop}\to B$ in $\cat{Coalg}$. We thus have, for $j\in J$, inclusion functors
    \begin{equation}\label{eq:iotaj}
      \tensor*[_{(I,J)}]{\cat{Halg}}{}
      \xrightarrow{\quad\iota_j\quad}
      (-)^{cop}\circ U_I\downarrow U_I
      :=
      \text{\emph{comma category} of \cite[Notation 2.42]{ar}}
    \end{equation}
    for the forgetful functor \eqref{eq:ibialg2coalg} analogous to \eqref{eq:fgt2coalg}. On the one hand the comma category is accessible by \cite[Theorem 2.43]{ar}, while on the other hand, as in \cite[Remark 45]{porst_formal-1}, the functors \eqref{eq:iotaj} realize their common domain $\tensor*[_{(I,J)}]{\cat{Halg}}{}$ as a joint equifier. Accessibility once more follows from \cite[Lemma 2.76]{ar}.

  \item\label{item:th:ijh.loc.pres:pf.locpres}: Local presentability. This follows from \cite[Corollary 2.47]{ar} and the preceding steps: in the presence of either completeness of cocompleteness, local presentability is equivalent to accessibility. 
    
  \item\label{item:th:ijh.loc.pres:pf.monad}: Monadicity along \eqref{eq:fgt2coalg}. We already know that $U_{(I,J)}$ is continuous from step \ref{item:th:ijh.loc.pres:pf.crlim}. It is also easily seen to preserve \emph{directed} \cite[\S 1.A, pp.8-9]{ar} colimits, as these are computable at the vector-space level. Given accessibility, it follows \cite[Theorem 1.66]{ar} that the functor is a right adjoint. It also
    \begin{itemize}
    \item \emph{reflects isomorphisms} \cite[\S V.7]{mcl_2e}, for these are simply the bijective morphisms in both categories of \eqref{eq:fgt2coalg};

    \item and preserves coequalizers for those parallel arrows $f$ and $g$ for which $U_{(I,J)}f$ and $U_{(I,J)}g$ has a \emph{split coequalizer} \cite[Definition 4.4.2]{brcx_hndbk-2} in $\cat{Coalg}$, as in the proof of \cite[\S 2.6, Proposition]{zbMATH05312006}.
    \end{itemize}
    Monadicity follows from (one version of) \emph{Beck's theorem} \cite[Theorem 4.4.4]{brcx_hndbk-2}.

  \item\label{item:th:ijh.loc.pres:pf.monad.bis}: Monadicity along \eqref{eq:fgt2bialg}. We focus on the $J'=\emptyset$ variant of the claim for simplicity:
    \begin{equation*}
      \tensor*[_{(I,J)}]{\cat{Halg}}{}
      \xrightarrow{\quad U_{(,J)}:=\cat{forget}\quad}
      \tensor*[_I]{\cat{Bialg}}{};
    \end{equation*}
    the argument will go through in general.

    This will be another application of Beck's characterization of monadicity. Note first that isomorphism reflection and the preservation of sufficiently well-behaved coequalizers (\cite[conditions (2)(b) and (2)(c) of Theorem 4.4.4]{brcx_hndbk-2}) hold for the left-hand arrow in
    \begin{equation*}
      \begin{tikzpicture}[>=stealth,auto,baseline=(current  bounding  box.center)]
        \path[anchor=base] 
        (0,0) node (l) {$\tensor*[_{(I,J)}]{\cat{Halg}}{}$}
        +(3,.5) node (u) {$\tensor*[_I]{\cat{Bialg}}{}$}
        +(6,0) node (r) {$\cat{Coalg}$}
        +(3,0) node {$\Downarrow\cong$}
        ;
        \draw[->] (l) to[bend left=6] node[pos=.5,auto] {$\scriptstyle U_{(,J)}$} (u);
        \draw[->] (u) to[bend left=6] node[pos=.5,auto] {$\scriptstyle U_I$} (r);
        \draw[->] (l) to[bend right=10] node[pos=.5,auto,swap] {$\scriptstyle U_{(I,J)}$} (r);
      \end{tikzpicture}
    \end{equation*}
    because they do for the other two arrows: Remark~\ref{res:digr}\ref{item:res:digr:bialg2coalg} for the right-hand functor and item~\ref{item:th:ijh.loc.pres:pf.monad} above for the bottom composition. As to \eqref{eq:fgt2bialg} being a right adjoint (\cite[Theorem 4.4.4(2)(a)]{brcx_hndbk-2}), it follows from \cite[Theorem 1.66]{ar}: it is a continuous functor between locally presentable categories, preserving filtered colimits. 
  \end{enumerate}
  This completes the proof of the theorem. 
\end{proof}

In particular, by \cite[Remarks 1.56(1) and 1.56(2)]{ar}:

\begin{corollary}\label{cor:ijh.co.compl}
  The categories $\tensor*[_{(I,J)}]{\cat{Halg}}{}$ are complete and cocomplete.  \qedhere
\end{corollary}

\section{On and around the category of Hopf braces}\label{se:hbr}

We now focus more closely on $\cat{Hbr}$ and its cocommutative counterpart $\cat{Hbr}^{coc}$. We have the following brace-specific variant of Theorem \ref{th:ijh.loc.pres}. 

\begin{theorem}\label{th:hbr.access}
  \begin{enumerate}[(1),wide]
  \item\label{item:th:hbr.access:hbr} The category $\cat{Hbr}$ is accessible. 

  \item\label{item:th:hbr.access:hbr.coc} Similarly, the category $\cat{Hbr}^{coc}$ is locally presentable and monadic along the forgetful functor
    \begin{equation*}
      \cat{Hbr}^{coc}
      \xrightarrow{\quad U:=\cat{forget}\quad}
      \cat{Coalg}^{coc}.
    \end{equation*}
  \end{enumerate}
\end{theorem}
\begin{proof}
  \begin{enumerate}[(I),wide]
  \item \textbf{: Accessibility in both \ref{item:th:hbr.access:hbr} and \ref{item:th:hbr.access:hbr.coc}.} The proof strategy is the same as in Theorem \ref{th:ijh.loc.pres} (and the various results in \cite{zbMATH05312006,porst_formal-1} that proof references). Focusing on $\cat{Hbr}$ to fix the notation, the defining equation \eqref{eq:hbr} realizes that category as the equifier of two natural transformations
    \begin{equation*}
      \begin{tikzpicture}[>=stealth,auto,baseline=(current  bounding  box.center)]
        \path[anchor=base] 
        (0,0) node (l) {$\tensor*[_2]{\cat{Halg}}{}$}
        +(6,0) node (r) {$\cat{Vect}$,}
        +(2,0) node (ld) {$\Downarrow$}
        +(4,0) node (rd) {$\Downarrow$}
        ;
        \draw[->] (l) to[bend left=20] node[pos=.5,auto] {$\scriptstyle \bullet^{\otimes 3}\circ\cat{forget}$} (r);
        \draw[->] (l) to[bend right=20] node[pos=.5,auto,swap] {$\scriptstyle \cat{forget}$} (r);
      \end{tikzpicture}
    \end{equation*}
    one given by the left-hand side of \eqref{eq:hbr} and one by the right-hand side. The conclusion follows from the known accessibility of $\tensor*[_2]{\cat{Halg}}{}$ (indeed, even local presentability: Theorem \ref{th:ijh.loc.pres}) and \cite[Lemma 2.76]{ar}. 

  \item \textbf{: Local presentability and adjointness in \ref{item:th:hbr.access:hbr.coc}.} For cocommutative 2-Hopf algebras both sides of \eqref{eq:hbr} constitute coalgebra morphisms $H^{\otimes 3}\to H$ (in general, only the left-hand side does). It follows that $\cat{Hbr}^{coc}$ is a \emph{variety of (finitary) algebras} \cite[\S 3.A]{ar} internal to the category of cocommutative coalgebras, and local presentability (given that of $\cat{Coalg}^{coc}$ \cite[\S 2.7, Proposition and Remark 1]{zbMATH05312006}) is no more difficult to prove than over $\cat{Set}$ \cite[Corollary 3.7]{ar}. By the same token, the forgetful functor $U$ is a right adjoint. 

  \item \textbf{: Monadicity.} As in step \ref{item:th:ijh.loc.pres:pf.monad} in the proof of Theorem \ref{th:ijh.loc.pres}: $U$ reflects isomorphisms (bijective morphisms on both sides) and preserves coequalizers for pairs whose $\cat{Coalg}^{coc}$-coequalizer is split.
  \end{enumerate}
\end{proof}

There is a consequent parallel to Corollary \ref{cor:ijh.co.compl}; the first statement follows from \cite[Remarks 1.56(1) and 1.56(2)]{ar}, while the second (given the accessibility ensured by Theorem \ref{th:hbr.access}\ref{item:th:hbr.access:hbr}) follows from \cite[Corollary 2.47]{ar}.

\begin{corollary}\label{cor:hbr.co.compl}
  \begin{enumerate}[(1),wide]
  \item The category $\cat{Hbr}^{coc}$ is complete and cocomplete.

  \item For $\cat{Hbr}$ the following conditions are equivalent.
    \begin{enumerate}[(a),wide]
    \item\label{item:cor:hbr.co.compl:compl} completeness;

    \item\label{item:cor:hbr.co.compl:cocompl} cocompleteness;

    \item\label{item:cor:hbr.co.compl:lp} local presentability.  \qedhere
    \end{enumerate}    
  \end{enumerate}
\end{corollary}

\begin{remark}
Corollary~\ref{cor:hbr.co.compl} provides the appropriate context to correct an inaccurate claim made in the literature by the first named author. \cite[Theorem 3.1]{AA} states that the category of Hopf braces is complete and although we do not know whether the statement itself is correct, we point out that there is a flaw in the proof of the aforementioned result. To be more precise, the construction of products provided in the proof of \cite[Theorem 3.1]{AA} works only for cocommutative Hopf braces and therefore it only implies the completeness of said category. 
\end{remark}

\section{Colimits and free objects in the category of cocommutative Hopf braces}\label{se:colim}

This section is devoted to some relatively explicit constructions for various adjoints and / or colimits involving Hopf braces and related categories, in the spirit of Takeuchi's construction \cite[\S 1]{zbMATH03344702} of the free Hopf algebra on a coalgebra, or Pareigis' description \cite[Theorem 2.6.3]{par_qg-ncg} of the free Hopf algebra on a bialgebra. 

We start by freely adjoining antipodes to a 2-bialgebra. The construction lies ``upstream'' to constructing coproducts of 2-Hopf algebras (which Corollary~\ref{cor:ijh.co.compl} ensures exist): applying a left adjoint of
\begin{equation*}
  \tensor*[_2]{\cat{Halg}}{}
  \xrightarrow{\quad\cat{forget}\quad}
  \tensor*[_2]{\cat{Bialg}}{}
\end{equation*}
to a coproduct of 2-bialgebras will produce a coproduct of 2-Hopf algebras. 

Note also that Proposition \ref{pr:ladj} below is a consequence of Theorem \ref{th:ijh.loc.pres}; the idea here is to give some sense of what the claimed adjoint functors look like. 

\begin{proposition}\label{pr:ladj}
  For sets $J'\subseteq J\subseteq I$ the forgetful functors
  \begin{eqnarray*}
  \tensor*[_{(I,J)}]{\cat{Halg}}{}
    &\xrightarrow{\quad\cat{forget}\quad}&
    \tensor*[_{(I,J')}]{\cat{Halg}}{}\\
       \tensor*[_{(I,J)}]{\cat{Halg}}{}^{coc}
    &\xrightarrow{\quad\cat{forget}\quad}&
    \tensor*[_{(I,J')}]{\cat{Halg}}{}^{coc}
  \end{eqnarray*}
  have left adjoints. 
\end{proposition}
\begin{proof}
  In order to avoid notational unpleasantness, we illustrate the construction by focusing on constructing the left adjoint to
  \begin{equation*}
    \tensor*[_{(2=\{0,1\},\{0\})}]{\cat{Halg}}{}
    \xrightarrow{\quad\cat{forget}\quad}
    \tensor*[_{(2,\emptyset)}]{\cat{Halg}}{}
    =
    \tensor*[_2]{\cat{Bialg}}{}.
  \end{equation*}
  The goal, in other words, is to append \emph{one} antipode freely to a 2-bialgebra. Note that we have identified $2$ with the set $\{0,1\}$ (as one commonly does in set-theoretic foundations \cite[Definition I.7.13]{kun_st_2nd_1983}); the $0^{th}$ and $1^{st}$ multiplication will be our old $\cdot$ and $\diamond$, and it is the former that must acquire an antipode. The construction will be recursive.
  \begin{itemize}[wide]
  \item Starting with a 2-bialgebra $(B,\cdot,\diamond)$, set $B_0:=B$ and construct the enveloping Hopf algebra $B_0\to B_1$ for $\cdot$ (ignoring $\diamond$ entirely), as in \cite[Theorem 2.6.3]{par_qg-ncg}. We abuse notation in denoting the multiplication of $B_1$ by $\cdot$ again. 

  \item Let $B_1\to B_2$ be the universal map from $(B_1,\cdot)$ into a 2-bialgebra $(B_2,\cdot,\diamond)$. In other words, this is the $B_1$-component of the unit of the adjunction built on the functor
    \begin{equation*}
      \tensor*[_2]{\cat{Bialg}}{}
      \xrightarrow{\quad\cat{forget}\quad}
      \cat{Bialg}
    \end{equation*}
    forgetting multiplication 1. 

  \item At this point we have a map $B_0\xrightarrow{\varphi} B_2$ between two 2-bialgebras, in principle respecting only multiplication 0. Consider the universal 2-bialgebra morphism $B_2\xrightarrow{\varphi'} B_3$ satisfying
    \begin{equation*}
      \begin{aligned}
        \varphi'(x\cdot y)
        &=\varphi'(x)\cdot \varphi'(y)
          ,\quad
          \forall x,y\in B_2
          \quad\text{and}\quad\\
        (\varphi'\varphi)(x\diamond y)
        &=(\varphi'\varphi)(x)\diamond (\varphi'\varphi)(y)
          ,\quad
          \forall x,y\in B_1.
      \end{aligned}      
    \end{equation*}
    By construction, the composition $B_0\xrightarrow{\varphi'\varphi}B_3$ is a 2-bialgebra morphism. 
    
  \item The entire procedure can now be reiterated, with $B_3$ in place of $B_0$: $B_3\to B_4$ is the universal map into a Hopf algebra with underlying multiplication $\cdot$, etc. 
  \end{itemize}
  That the resulting map $B=B_0\to \varinjlim_n B_n$ is the desired universal 2-bialgebra morphism into an object in $\tensor*[_{(2,\{0\})}]{\cat{Halg}}{}$ is immediate from the construction.
\end{proof}

Turning to cocommutative Hopf braces, we have the following analogous result.

\begin{proposition}\label{pr:incadj}
  The inclusion functor $I \colon \cat{Hbr}^{coc} \to \tensor*[_2]{\cat{Halg}}{^{coc}}$ has a left adjoint.
\end{proposition}
\begin{proof}
  Let $\bigl(H, \overline{\cdot},\overline{\diamond}, \overline{1}, \overline{\Delta}, \overline{\varepsilon}, \overline{S}, \overline{T}\bigl)$ be an object in $\tensor*[_2]{\cat{Halg}}{^{coc}}$ and define the $k$-linear maps $f$, $g \colon H^{\otimes 3} \to H$ on $x \otimes y \otimes z \in H^{\otimes 3}$ by
  \begin{align*}
    f(x \otimes y \otimes z) &= x\, \overline{\diamond}\, (y\, \overline{\cdot}\, z)\\
    g(x \otimes y \otimes z) &= \bigl(x_{(1)}\, \overline{\diamond}\, y\bigl)\, \overline{\cdot} \,\overline{S}(x_{(2)})\, \overline{\cdot} \, \bigl(x_{(3)} \, \overline{\diamond}\, z\bigl)
  \end{align*}
  The cocommutativity assumption on $H$ implies that both $f$ and $g$ are coalgebra maps and therefore the span $J_{0}$ of $\{f(x \otimes y \otimes z) - g(x \otimes y \otimes z) ~|~ x \otimes y \otimes z \in H^{\otimes 3}\}$ is in fact a coideal in $(H,\, \overline{\Delta},\, \overline{\varepsilon})$ (see, for example, \cite[Exercise 2.1.29]{rad}).

  Denoting by $J$ the Hopf brace ideal $\widetilde{J_0}\le H$ defined in Construction~\ref{con:id.chn}, the quotient space $H/J$ is again an object in $\tensor*[_2]{\cat{Halg}}{^{coc}}$ with the induced structures denoted by $\bigl(\cdot, \diamond, 1,\, \Delta,\, \varepsilon,\, S,\, T\bigl)$. Furthermore, the canonical projection $\pi_{H} \colon H \to H/J$ is a morphism in $\tensor*[_2]{\cat{Halg}}{^{coc}}$ and we obtain:
  \begin{align*}
    0 &= \pi_{H}\Bigl(x\, \overline{\diamond}\, (y\, \overline{\cdot}\, z) - \bigl(x_{(1)}\, \overline{\diamond}\, y\bigl)\, \overline{\cdot} \,\overline{S}(x_{(2)})\, \overline{\cdot} \, \bigl(x_{(3)} \, \overline{\diamond}\, z\bigl)\Bigl)\\
      &= \pi_{H}(x) \, \diamond\, (\pi_{H}(y)\, \cdot\, \pi_{H}(z)) - \bigl(\pi_{H}(x_{(1)})\, \diamond\, \pi_{H}(y)\bigl)\, \cdot \,S(\pi_{H}(x_{(2)}))\, \cdot\, \bigl(\pi_{H}(x_{(3)}) \,\diamond\, \pi_{H}(z)\bigl)
  \end{align*}
  for all $x$, $y$, $z \in H$. This shows that \eqref{eq:hbr} holds in $H/J$ and therefore $\bigl(H/J, \cdot, \diamond, 1, \Delta, \varepsilon, S, T\bigl)$ is a Hopf brace.
  
  Consider now $h \colon H \to B$ to be a morphism in $\tensor*[_2]{\cat{Halg}}{^{coc}}$, where $\bigl(B, \cdot_{B}, \diamond_{B}, 1, \Delta_{B}, \varepsilon_{B}, S_{B}, T_{B}\bigl)$ is a Hopf brace. Then $J_{0} \subset {\rm ker}\, h$; indeed, for all $x$, $y$, $z \in H$ we have:
  \begin{align*}
    & h\Bigl(x\, \overline{\diamond}\, (y\, \overline{\cdot}\, z) - \bigl(x_{(1)}\, \overline{\diamond}\, y\bigl)\, \overline{\cdot} \,\overline{S}(x_{(2)})\, \overline{\cdot} \, \bigl(x_{(3)} \, \overline{\diamond}\, z\bigl)\Bigl)\\
    & =  h(x)\, \diamond_{B}\, (h(y)\, \cdot_{B}\, h(z)) -  \bigl(h(x_{(1)})\, \diamond_{B}\, h(y)\bigl)\, \cdot_{B} \, S_{B}(h(x_{(2)}))\, \cdot_{B} \, \bigl(h(x_{(3)}) \, \diamond_{B}\, h(z)\bigl)\\
    & = 0
  \end{align*}
  where the last equality holds because $B$ is a Hopf brace. Therefore, $J \subset {\rm ker}\, h$ and there exists a unique morphism $\overline{h} \colon H/J \to B$ in $\tensor*[_2]{\cat{Halg}}{^{coc}}$ (and in the full subcategory $\cat{Hbr}^{coc}$) such that the following diagram is commutative:
  \begin{eqnarray*}
    \xymatrix {&  {H} \ar[dr]_{h}\ar[r]^-{\pi_{H}} & {H/J}\ar[d]^-{\overline{h}}\\
               &  {} & {B}}
  \end{eqnarray*}
  To conclude, the pair $(H/J, \pi)$ is a universal arrow from $H$ to $I$ and therefore $I$ has a left adjoint by (\cite[\S III.1, \S IV Theorem 2]{mcl_2e}). 
\end{proof}

We conclude the discussion on free objects with the following straightforward consequence of Proposition~\ref{pr:ladj} and Proposition~\ref{pr:incadj}:

\begin{corollary}
Both forgetful functors $U \colon \cat{Hbr}^{coc} \to \cat{Halg}^{coc}$ have left adjoints.
\end{corollary}

We next turn to explicitly describing colimits in $\cat{Hbr}^{coc}$ (whose existence is guaranteed by Corollary~\ref{cor:ijh.co.compl}). In light of (the dual of) \cite[\S V.2, Theorem 1]{mcl_2e}, we restrict to describing coequalizers and coproducts.

\subsection*{Coequalizers in the category of (cocommutative) Hopf braces}

Let $f$, $g \colon A \to B$ be two morphisms of Hopf braces and consider $I_{0}$ to be the vector space generated by the set $\{f(a) - g(a) ~|~ a \in A\}$. It can be seen by a simple computation that $I_{0}$ is in fact a coideal in $B$ (see, for example, \cite[Exercise 2.1.29]{rad}) and that $S(I_{0}) \subseteq I_{0}$ and $T(I_{0}) \subseteq I_{0}$, where $S$ and $T$ denote the antipodes of the first and the second Hopf algebra structures of $B$, respectively.


Consider the Hopf brace $B/I$ that is the quotient of $B$ by the Hopf brace ideal $I:=\widetilde{I_0}$ of Construction~\ref{con:id.chn}. The proof will be finished once we show that the pair $\bigl(B/I, \pi\bigl)$ is the coequalizer of $f$ and $g$, where $\pi \colon B \to B/I$ is the projection map. Indeed, let $h \colon B \to H$ be another Hopf brace morphism such that $h \circ f = h \circ g$. This shows that $I_{0} \subset {\rm ker}\, h$ and since $h$ is a Hopf brace morphism it can be easily proved using an inductive argument that $I_{n} \subset {\rm ker}\, h$ for all $n \in \NN^{*}$. Therefore $I \subset {\rm ker}\, h$ and we obtain a unique morphism of Hopf braces $h' \colon B/I \to H$ such that $h' \circ \pi = h$, as desired.

\subsection*{Coproducts in the category of cocommutative Hopf braces}

The construction of coproducts in $\cat{Hbr}^{coc}$ will be given in steps starting with the corresponding constructions being performed first in the categories $\tensor*[_2]{\cat{Bialg}}{}^{coc}$ and $\tensor*[_2]{\cat{Halg}}{^{coc}}$.
\begin{enumerate}[(1),wide]
\item Coproducts in $\tensor*[_2]{\cat{Bialg}}{}$ and $\tensor*[_2]{\cat{Bialg}}{}^{coc}$.
  
  Consider $\bigl(B_{l},\, \cdot_{l}, \diamond_{l}, 1_{l}, \Delta_{l},\, \varepsilon_{l}\bigl)_{l \in I}$ to be a family of objects in $\tensor*[_2]{\cat{Bialg}}{}$. To start with, denote by $\Bigl(\bigl(\bigoplus_{l \in I} B_{l},\, \overline{\Delta},\, \overline{\varepsilon}\bigl),\, (j_{l})_{l \in I}\Bigl)$ the coproduct in $\cat{Coalg}$ of the underlying coalgebras, where $\bigoplus_{l \in I} B_{l}$ denotes the direct sum of this family in ${\cat{VECT}}$ and $j_{t} \colon B_{t} \to \bigoplus_{l \in I} B_{l}$ are the canonical injections for all $t \in I$. Recall (\cite[Exercises 2.1.18 and 2.1.35 ]{rad} that the comultiplication $\overline{\Delta}$ and the counit $\overline{\varepsilon}$ are the unique linear maps such that the following hold for all $t \in I$:
  \begin{align*}
    \overline{\Delta} \circ j_{t} &= (j_{t} \otimes j_{t})\circ \Delta_{t} \\
    \overline{\varepsilon} \circ j_{t} &= \varepsilon_{t}
  \end{align*}
  Now let $\Bigl(\bigl(F\bigl(\bigoplus_{l \in I}B_{l}\bigl),\, \cdot, \diamond, 1\bigl),\, i\Bigl)$ denote the free $2$-algebra on the vector space $\bigoplus_{l \in I} B_{l}$ (\cite[Section 5]{LR}), where $i\colon \bigoplus_{l \in I}B_{l} \rightarrow F\bigl(\bigoplus_{l \in I}B_{l}\bigl)$ is a linear map. Furthermore, it can be easily seen, as in \cite[Theorem 5.3.1]{rad}, that the free $2$-algebra $F\bigl(\bigoplus_{l \in I}B_{l}\bigl)$ is in fact an object in $\tensor*[_2]{\cat{Bialg}}{}$ with comultiplication $\Delta$ and counit $\varepsilon$ given by the unique morphisms in $\tensor*[_2]{\cat{Alg}}{}$ such that the following hold:  
  \begin{align}
    \Delta \circ i &= (i \otimes i) \circ \overline{\Delta} \label{eq4.1}\\
    \varepsilon \circ i &= \overline{\varepsilon} \label{eq4.2}
  \end{align}
  In particular,~(\ref{eq4.1}) and~(\ref{eq4.2}) imply that $i\colon \bigoplus_{l \in I}B_{l} \rightarrow F\bigl(\bigoplus_{l \in I}B_{l}\bigl)$ is a coalgebra map. Throughout, for all $t \in I$, we denote by $u_{t}$ the coalgebra map $i \circ j_{t} \colon \bigl(B_{t},\, \Delta_{t}, \, \varepsilon_{t}\bigl) \to \Bigl(F\bigl(\bigoplus_{l \in I} B_{l}\bigl),\, \Delta,\, \varepsilon\Bigl)$. Consider now the vector space $J_{0}$ generated by the set $J \cup J' \cup J''$, where:
  \begin{eqnarray*}
    J &=& \{1 - u_{t}(1_{t})  ~ |~ t \in I \}\\
    J' &=& \{ u_{t}(x\, \cdot _{t} \,x')-u_{t}(x)\, \cdot\, u_{t}(x') ~ |~
           x, x' \in B_{t};\, t \in I \}\\
    J'' &=& \{ u_{t}(x\, \diamond _{t} \,x')- u_{t}(x)\, \diamond\, u_{t}(x') ~ |~
            x, x' \in B_{t};\, t \in I \}
  \end{eqnarray*}
  As $u_{t}$ is a coalgebra map and $B_{t}$ is a bialgebra for all $t \in I$ and $F\bigl(\bigoplus_{l \in I} B_{l}\bigl)$ is in $\tensor*[_2]{\cat{Bialg}}{}$, it can be easily seen that $J_{0}$ is in fact a coideal in $\Bigl(F\bigl(\bigoplus_{l \in I} B_{l}\bigl),\, \Delta, \, \varepsilon\Bigl)$. Next, we construct a family of coideals in $\Bigl(F\bigl(\bigoplus_{l \in I} B_{l}\bigl),\, \Delta, \, \varepsilon\Bigl)$ as follows:
  \begin{itemize}[wide]
  \item $J_{1}$ is the ideal generated by $J_{0}$ in $\Bigl(F
    \bigl(\bigoplus_{l \in I}B_{l}\bigl), \cdot, 1\Bigl)$
  \item $J_{2}$ is the ideal generated by $J_{1}$ in $\Bigl(F
    \bigl(\bigoplus_{l \in I}B_{l}\bigl), \diamond, 1\Bigl)$
  \item $J_{3}$ is the ideal generated by $J_{2}$ in $\Bigl(F
    \bigl(\bigoplus_{l \in I}B_{l}\bigl), \cdot, 1\Bigl)$
  \end{itemize}
  Continuing this process we obtain a chain of coideals:
  \begin{eqnarray*}
    J_{0} \subset  J_{1} \subset J_{2} \subset J_{3} \subset \cdots \subset J_{n} \subset \cdots
  \end{eqnarray*}
  The union $L = \cup_{i \in \NN}\, J_{i}$ can be easily seen to be a coideal in $\Bigl(F\bigl(\bigoplus_{l \in I} B_{l}\bigl),\, \Delta, \, \varepsilon\Bigl)$ and by construction is an ideal in both $\Bigl(F
  \bigl(\bigoplus_{l \in I}B_{l}\bigl), \cdot, 1\Bigl)$ and $\Bigl(F
  \bigl(\bigoplus_{l \in I}B_{l}\bigl), \diamond, 1\Bigl)$. Hence, the quotient space $H = F\bigl(\bigoplus_{l \in I}B_{l}\bigl)/L$ remains an object in $\tensor*[_2]{\cat{Bialg}}{}$ which, furthermore, is the coproduct of the family $\bigl(B_{l},\, \cdot_{l},\, \diamond_{l}\,  1_{l},\, \Delta_{l},\,
  \varepsilon_{l}\bigl)_{l \in I}$ in $\tensor*[_2]{\cat{Bialg}}{}$.

  Finally, the coproduct construction in the category $\tensor*[_2]{\cat{Bialg}}{}^{coc}$ goes in the same vein as the one performed above in $\tensor*[_2]{\cat{Bialg}}{}$. Indeed, the cocommutativity of the family $\bigl(B_{l},\, \cdot_{l},\,\diamond_{l},\,  1_{l},\, \Delta_{l},\,
  \varepsilon_{l}\bigl)_{l \in I}$ of objects in $\tensor*[_2]{\cat{Bialg}}{}^{coc}$ implies cocommutativity of both $\bigoplus_{l \in I} B_{l}$ and $F\bigl(\bigoplus_{l \in I} B_{l}\bigl)$.

\item Coproducts in $\tensor*[_2]{\cat{Halg}}{}$ and $\tensor*[_2]{\cat{Halg}}{^{coc}}$.
  
  Let now $\bigl(H_{l},\, \cdot_{l}, \diamond_{l}, 1_{l}, \Delta_{l},\, \varepsilon_{l},\, S_{l},\, T_{l}\bigl)_{l \in I}$ be a family of objects in $\tensor*[_2]{\cat{Halg}}{}$ and denote the coproduct in $\tensor*[_2]{\cat{Bialg}}{}$ of the underlying $2$-bialgebras by $\Bigl(\bigl(B,\, \cdot_{B}, \diamond_{B}, 1_{B}, \Delta_{B},\,
  \varepsilon_{B}\bigl), (q_{l})_{l \in I}\Big)$. Given the existence of a left adjoint for the forgetful functor $U \colon \tensor*[_2]{\cat{Halg}}{} \to \tensor*[_2]{\cat{Bialg}}{}$ (constructed in Proposition~\ref{pr:ladj}), we have a universal arrow from $B$ to $U$ (\cite[\S III.1, \S IV Theorem 2]{mcl_2e}), say $i \colon B \to H$ where $\bigl(H,\, \cdot_{H}, \diamond_{H}, 1_{H}, \Delta_{H},\, \varepsilon_{H}, S_{H}, T_{H}\bigl)$ is a $2$-Hopf algebra and $i$ is a morphism in $\tensor*[_2]{\cat{Bialg}}{}$. Then $\Bigl(\bigl(H,\, \cdot_{H}, \diamond_{H}, 1_{H}, \Delta_{H},\,
  \varepsilon_{H}, S_{H}, T_{H}\bigl), (i \circ q_{l})_{l \in I}\Bigl)$ is the coproduct in $\tensor*[_2]{\cat{Halg}}{}$ of the family $\bigl(H_{l},\, \cdot_{l}, \diamond_{l}, 1_{l}, \Delta_{l},\,
  \varepsilon_{l},\, S_{l},\, T_{l}\bigl)_{l \in I}$. To start with, given $l \in I$, note that $i \circ q_{l} \colon H_{l} \to H$ is a $2$-Hopf algebra map as a consequence of being a $2$-bialgebra map between the $2$-Hopf algebras $H_{l}$ and $H$.
  Furthermore, assume $(L, , \cdot_{L}, \diamond_{L}, 1_{L}, \Delta_{L},\, \varepsilon_{L}, S_{L}, T_{L})$ is another $2$-Hopf algebra and $(s_{l} \colon H_{l} \to L)_{l \in I}$ is a family of $2$-Hopf algebra morphisms. The universal property of the coproduct in $\tensor*[_2]{\cat{Bialg}}{}$ yields a unique $2$-bialgebra map $\overline{f} \colon B \to L$ such that $\overline{f} \circ q_{l} = s_{l}$ for all $l \in I$. Similarly, as $i$ is a universal map from $B$ to $U$, we obtain a unique $2$-Hopf algebra morphism $f \colon H \to L$ such that $f \circ i = \overline{f}$. To conclude, $f \colon H \to L$ is the unique $2$-Hopf algebra morphism such that $f \circ i \circ q_{l} = s_{l}$ for all $l \in I$.
  
  The same strategy for constructing coproducts applies verbatim to cocommutative $2$-Hopf algebras.
  
\item Coproducts in $\cat{Hbr}^{coc}$.

  Let $\bigl(H_{l}, \cdot_{l}, \diamond_{l}, 1, \Delta_{l},
  \varepsilon_{l}, S_{l}, T_{l}\bigl)_{l \in I}$ be a family of cocommutative Hopf
  braces and consider the coproduct, say $\Bigl(\bigl(H, \overline{\cdot},\overline{\diamond}, \overline{1}, \overline{\Delta}, \overline{\varepsilon}, \overline{S}, \overline{T}\bigl),\, (q_{l})_{l \in I}\Bigl)$, of this family in $\tensor*[_2]{\cat{Halg}}{^{coc}}$. Next, for all $t$, $s$, $r \in I$, we denote by $f_{t, s, r}$, $g_{t, s, r} \colon H_{t}\otimes H_{s} \otimes H_{r} \to H$ the $k$-linear maps defined as follows for all $x \otimes y \otimes z \in H_{t}\otimes H_{s} \otimes H_{r} $:
  \begin{align*}
    f_{t, s, r}(x \otimes y \otimes z) &= q_{t}(x)\, \overline{\diamond}\, (q_{s}(y)\, \overline{\cdot}\, q_{r}(z))\\
    g_{t, s, r}(x \otimes y \otimes z) &= \bigl(q_{t}(x_{(1)})\, \overline{\diamond}\, q_{s}(y)\bigl)\, \overline{\cdot} \,\overline{S}(q_{t}(x_{(2)}))\, \overline{\cdot} \, \bigl(q_{t}(x_{(3)}) \, \overline{\diamond}\, q_{r}(z)\bigl)
  \end{align*}
  Note that the cocommutativity assumption implies that both $f_{t, s, r}$ and $g_{t, s, r}$ are coalgebra maps and therefore the linear span $J_{0}$ of $\{f_{t, s, r}(x \otimes y \otimes z) - g_{t, s, r}(x \otimes y \otimes z) ~|~ x \otimes y \otimes z \in H_{t}\otimes H_{s} \otimes H_{r}; t,\, s,\, r \in I\}$ is in fact a coideal in $(H,\, \overline{\Delta},\, \overline{\varepsilon})$ (again by \cite[Exercise 2.1.29]{rad}).

  Denote by $\bigl(\cdot,\,\diamond, 1, \Delta, \varepsilon, S, T\bigl)$ the Hopf-brace structure on the quotient $H/J$ by the ideal $J:=\widetilde{J_0}$ described in Construction~\ref{con:id.chn}. The proof will be finished once we show that
  \begin{equation*}
    \Bigl(\bigl(H/J, \cdot, \diamond, 1, \Delta, \varepsilon, S, T),\, (s_{l})_{l \in I} \Bigl)
    \cong
    \coprod_{l} \bigl(H_{l}, \cdot_{l}, \diamond_{l}, 1, \Delta_{l}, \varepsilon_{l}, S_{l}, T_{l}\bigl)_{l \in I}
    \in    
    \cat{Hbr}^{coc},
  \end{equation*}
  where $s_{t} = \pi \circ q_{t}$ for all $t \in I$ and $\pi \colon H \to H/J$ denotes the canonical projection. To this end, consider another object $\bigl(B, \cdot_{B}, \diamond_{B}, 1_{B}, \Delta_{B}, \varepsilon_{B}, S_{B}, T_{B}\bigl)$ in $\cat{Hbr}^{coc}$ and a family of morphisms $\bigl(b_{l} \colon H_{l} \to B\bigl)_{l \in I}$ in $\cat{Hbr}^{coc}$. As $\bigl(B, \cdot_{B}, \diamond_{B}, 1_{B}, \Delta_{B}, \varepsilon_{B}, S_{B}, T_{B}\bigl)$ is in particular an object in $\tensor*[_2]{\cat{Halg}}{^{coc}}$, there exists a unique morphism $v \colon H \to B$ in $\tensor*[_2]{\cat{Halg}}{^{coc}}$ such that $v \circ q_{t} = b_{t}$ for all $t \in I$. Now since $\bigl(B, \cdot_{B}, \diamond_{B}, 1_{B}, \Delta_{B}, \varepsilon_{B}, S_{B}, T_{B}\bigl)$ is a Hopf brace, the following holds for all $t$, $s$, $r \in I$ and $x \otimes y \otimes z \in H_{t}\otimes H_{s} \otimes H_{r} $:
  \begin{align*}
    & v\bigl(f_{t, s, r}(x \otimes y \otimes z) - g_{t, s, r}(x \otimes y \otimes z) \bigl)\\
    & = v\Bigl(q_{t}(x)\, \overline{\diamond}\, (q_{s}(y)\, \overline{\cdot}\, q_{r}(z)) -  \bigl(q_{t}(x_{(1)})\, \overline{\diamond}\, q_{s}(y)\bigl)\, \overline{\cdot} \,\overline{S}(q_{t}(x_{(2)}))\, \overline{\cdot} \, \bigl(q_{t}(x_{(3)}) \, \overline{\diamond}\, q_{r}(z)\bigl)\Bigl) \\
    & =  b_{t}(x)\, \diamond_{B}\, (b_{s}(y)\, \cdot_{B}\, b_{r}(z)) -  \bigl(b_{t}(x_{(1)})\, \diamond_{B}\, b_{s}(y)\bigl)\, \cdot_{B} \, S(b_{t}(x_{(2)}))\, \cdot_{B} \, \bigl(b_{t}(x_{(3)}) \, \diamond_{B}\, b_{r}(z)\bigl)\\
    & = 0
  \end{align*}
  This shows that $J_{0} \subset {\rm ker}\, v$ and since $v$ is a morphism in $\tensor*[_2]{\cat{Halg}}{^{coc}}$ we can conclude that $J \subset {\rm ker}\, v$. Therefore we have a unique morphism $w \colon H/J \to B$ in $\tensor*[_2]{\cat{Halg}}{^{coc}}$ such that $w \circ \pi = v$. Putting everything together, we obtain that $w$ is the unique morphism in $\cat{Hbr}^{coc}$ such that $w \circ s_{t} = b_{t}$ for all $t \in I$, which concludes the proof.
\end{enumerate}


\addcontentsline{toc}{section}{References}

\Addresses
\end{document}